\documentclass[11pt,twoside]{article}

\usepackage{latexsym}
\usepackage{amsmath}
\usepackage{amsthm}
\usepackage{amssymb}
\usepackage{vmargin}
\usepackage{stmaryrd}
\usepackage{euscript}
\usepackage{mathrsfs}
\usepackage{amscd}
\usepackage[all]{xy}
\usepackage{xr}
\DeclareMathAlphabet{\mathpzc}{OT1}{pzc}{m}{it}

\setmargins{32mm}{20mm}{14.6cm}{22cm}{1cm}{1cm}{1cm}{1cm}

\newtheorem{theorem}{Theorem}[section]
\newtheorem{proposition}[theorem]{Proposition}

\newtheorem{lemma}[theorem]{Lemma}
\newtheorem*{theorem*}{Theorem}
\newtheorem*{proposition*}{Proposition}
\newtheorem*{corollary*}{Corollary}
\newtheorem*{lemma*}{Lemma}
\newtheorem*{conjecture*}{Conjecture}

\theoremstyle{definition}
\newtheorem{definition}[theorem]{Definition}

\newtheorem*{definition*}{Definition}

\theoremstyle{remark}

\newtheorem{examples}[theorem]{Examples}
\newtheorem{remark}[theorem]{Remark}
\newtheorem{remarks}[theorem]{Remarks}

\newtheorem*{example*}{Example}
\newtheorem*{examples*}{Examples}
\newtheorem*{remark*}{Remark}
\newtheorem*{remarks*}{Remarks}
\newtheorem*{exercise*}{Exercise}

\newcommand\da{\!\downarrow\!}
\newcommand\ra{\rightarrow}
\newcommand\la{\leftarrow}

\newcommand\ten{\otimes}
\newcommand\vareps{\varepsilon}

\renewcommand\H{\mathrm{H}}
\newcommand\z{\mathrm{Z}}

\newcommand\Z{\mathbb{Z}}

\newcommand\R{\mathbb{R}}
\newcommand\Cx{\mathbb{C}}

\newcommand\vv{\mathbb{V}}

\newcommand\bA{\mathbb{A}}

\newcommand\bF{\mathbb{F}}
\newcommand\bG{\mathbb{G}}
\newcommand\bH{\mathbb{H}}

\newcommand\bP{\mathbb{P}}

\newcommand\bR{\mathbb{R}}

\newcommand\bT{\mathbb{T}}

\newcommand\C{\mathcal{C}}

\newcommand\cH{\mathcal{H}}

\newcommand\cS{\mathcal{S}}

\renewcommand\O{\mathscr{O}}

\newcommand\sA{\mathscr{A}}
\newcommand\sB{\mathscr{B}}

\newcommand\sG{\mathscr{G}}

\newcommand\sO{\mathscr{O}}
\newcommand\sP{\mathscr{P}}
\newcommand\sQ{\mathscr{Q}}

\newcommand\sU{\mathscr{U}}
\newcommand\sV{\mathscr{V}}

\newcommand\Def{\mathfrak{Def}}

\newcommand\fM{\mathfrak{M}}
\newcommand\fN{\mathfrak{N}}

\newcommand\fS{\mathfrak{S}}
\newcommand\fT{\mathfrak{T}}

\renewcommand\L{\Lambda}

\newcommand\m{\mathfrak{m}}

\newcommand\Ho{\mathrm{Ho}}

\newcommand\Alg{\mathrm{Alg}}

\newcommand\Aut{\mathrm{Aut}}

\newcommand\im{\mathrm{Im\,}}

\newcommand\Ob{\mathrm{Ob}\,}

\newcommand\mal{\mathrm{Mal}}
\newcommand\univ{\mathrm{univ}}

\newcommand\Spec{\mathrm{Spec}\,}

\newcommand\Set{\mathrm{Set}}

\newcommand\Grp{\mathrm{Grp}}
\newcommand\Aff{\mathrm{Aff}}

\newcommand\ad{\mathrm{ad}}

\newcommand\into{\hookrightarrow}

\newcommand\xra{\xrightarrow}

\newcommand\pr{\mathrm{pr}}

\newcommand\dmd{\diamond}
\newcommand\bt{\bullet}
\newcommand\by{\times}
\newcommand\mcl{\mathrm{MC}_L}
\newcommand\mc{\mathrm{MC}}

\newcommand\Rees{\mathrm{Rees}}

\newcommand\SL{\mathrm{SL}}
\newcommand\GL{\mathrm{GL}}

\newcommand\an{\mathrm{an}}

\newcommand\pd{\partial}
\newcommand\dc{d^{\mathrm{c}}}
\newcommand\half{\frac{1}{2}}

\newcommand\MHS{\mathrm{MHS}}
\newcommand\MTS{\mathrm{MTS}}

\newcommand\ugr{\underline{\mathrm{gr}}}

\newcommand\gpd{\mathrm{Gpd}}
\newcommand\Gpd{\mathrm{Gpd}}

\newcommand\red{\mathrm{red}}

\newcommand\row{\mathrm{row}}

\newcommand\dR{\mathrm{dR}}

\newcommand\oR{\mathbf{R}}

\begin{document}
\title{Hodge structures on analytic moduli of real pluriharmonic bundles}
\author{J.P.Pridham\thanks{This work was supported by Trinity College, Cambridge; and the Engineering and Physical Sciences Research Council [grant number  EP/F043570/1].}}
\maketitle

\begin{abstract}
We define and construct the real analytic moduli stack of pluriharmonic bundles on a compact K\"ahler manifold $X$, and show how this is equipped with Hodge and quaternionic structures. This stack maps to the de Rham moduli stack, giving rise to preferred sections of the Deligne-Hitchin twistor stack. We then show how the non-abelian mixed Hodge structures on Malcev homotopy types $X$ can be extended to  objects over the pluriharmonic moduli stack.
\end{abstract}

\tableofcontents

\section*{Introduction}
\addcontentsline{toc}{section}{Introduction}

This paper is a sequel to \cite{mhs}, in which we showed how to put mixed Hodge and mixed twistor structures on schematic (or, more generally, relative Malcev) homotopy types of compact K\"ahler manifolds. 

In \cite{Sim1} and \cite{Sim2}, Simpson defined the coarse  Betti, de Rham and Dolbeault  moduli spaces of a smooth projective complex variety. These  are all algebraic spaces, and he constructed a complex analytic isomorphism between the Betti and de Rham moduli spaces, then a homeomorphism between the de Rham and Dolbeault moduli spaces. The key to the latter was the correspondence between semisimple local systems, pluriharmonic bundles and Higgs bundles.

The de Rham moduli stack $\fM_{\dR,X,n}$ parametrises pairs $(\sV, \nabla)$ on $X$, where $\sV$ is a complex rank $n$ $\C^{\infty}$ vector bundle on $X$, and $\nabla$ a flat connection. This is a complex analytic stack, equipped with an antiholomorphic involution $\tau$, corresponding to complex conjugation of $\sV$. It extends to a stack $\fT_{\dR,X,n}$ over $\Cx\ten_{\R}\Cx-\{0\}$, 
which we will call the de Rham twistor stack. The fibre over $\alpha \ten 1 + \beta \ten i$ is given by pairs $(\sV, \nabla)$, where we now take $\nabla$ to be a flat $(\alpha d +\beta \dc)$-connection.
Thus the fibre of $\fT_{\dR,X,n}$ over $1\ten 1$ is just $\fM_{\dR,X,n}$. 

Moreover, the stack $\fT_{\dR,X,n}$ has an analytic action of $(\Cx\ten_{\R}\Cx)^* $ over $\Cx\ten_{\R}\Cx-\{0\}$, with $(\alpha \ten 1 + \beta \ten i)\dmd(\sV, \nabla):= (\sV,\alpha \nabla + \beta\ad_J\nabla)$. As explained in Remark \ref{twmhs}, this means that $\fT_{\dR,X,n}$ is a non-abelian Hodge filtration on   $\fM_{\dR,X,n}$. Moreover, taking the quotient by $\Cx^*$ gives us a stack $[\fT_{\dR,X,n}/\Cx^*]$ over $\bP^1(\Cx)\cong (\Cx\ten_{\R}\Cx-\{0\})/\Cx^*$. This is precisely the  Deligne-Hitchin twistor stack, which is usually defined in terms of gluing rather than as a quotient (Remark \ref{cfdh}).

The pluriharmonic moduli stack $\fM_{\cH,X,n}$ is defined over $\fM_{\dR,X,n}$, with the fibre over $(\sV, \nabla)$ parametrising hermitian metrics on $\sV$ with respect to which the connection $\nabla$ is pluriharmonic. This stack is real (not complex) analytic, and has an analytic action of the circle group $U_1$ (Definition \ref{clubdef}), which we can regard as giving it a pure Hodge structure of weight $0$. The pluriharmonic twistor stack   $\fT_{\cH,X,n}$ is an extension of $\fM_{\cH,X,n}$  over  the non-zero quaternions $\bH^*$. This stack admits a compatible multiplication by $\bH^*$, making it isomorphic to the trivial bundle $\fM_{\cH,X,n}\by \bH^*$.

There is a morphism 
$\Upsilon:\fT_{\cH,X,n}\to \fT_{\dR,X,n}$ of stacks given by forgetting the pluriharmonic structure, and this lies over a standard isomorphism $\Cx\ten_{\R}\Cx \cong \bH$. Thus $\fM_{\cH,X,n}$ gives sections of the map $\fT_{\dR,X,n}\to\Cx\ten_{\R}\Cx-\{0\}$.
Many of the known structures on the Deligne-Hitchin twistor space come from combining the actions of  $(\Cx\ten_{\R}\Cx)^*$ and $\bH^*$ with the injectivity of $\Upsilon$ on objects (see Remark \ref{cfsimpson} for one example).

Rather than just working with vector bundles, we generalise to work with $R$-bundles, with $R$ any real reductive pro-algebraic group (or groupoid) equipped with a Cartan involution. It then makes sense to consider relative Malcev homotopy types as in \cite{htpy}. There is a natural object over the stack $\fM_{\dR,X,R}$, whose fibre over the point $[\rho]$ (associated to a representation $\rho: \pi_fX \to R(\Cx)$ of the fundamental groupoid) is the relative Malcev homotopy type $X^{\rho, \mal}$. We denote the pullback of this universal object to $\fM_{\cH, X, R}$ by $X^{R, \univ}$, and show that this admits a natural algebraic mixed twistor structure  (Proposition \ref{univmts}). We then show that this can be enhanced to an analytic mixed Hodge structure (Proposition \ref{univmtsen}), extending the weight $0$ Hodge structure on  $\fM_{\cH, X, R}$. Taking the fibre at a point $[\rho]$ yields the mixed twistor and mixed Hodge structures on $X^{\rho, \mal}$ of \cite{mhs} (Remarks \ref{cfmhs}).

The structure of the paper is as follows.

In Sections \ref{dR} and \ref{tdR}, we define the de Rham  moduli and twistor stacks $\fM_{\dR,X,R}$ and $\fT_{\dR,X,R}$ of $R$-bundles on $X$, and construct the $(\Cx\ten_{\R}\Cx)^* $-action on the twistor stack. Sections \ref{cH} and \ref{tcH} contain the pluriharmonic analogues with these, including the $\bH^*$-action on $\fT_{\dR,X,n}$, together with properties of the forgetful functor $\Upsilon:\fT_{\cH,X,R}\to \fT_{\dR,X,R}$. In Section \ref{local}, the local structures of the twistor stacks are described, including that singularities of $\fT_{\dR,X,R}$ in the image of $\Upsilon$ are all quadratic  (Proposition \ref{cfgm}), which is a generalisation of  Goldman and Millson's corresponding result (\cite{GM}) for $\fM_{\dR,X,R}$.

The universal Malcev homotopy type $X^{R, \univ}$ over $\fM_{\cH,X,R} $ is defined in Section \ref{relhtpy}, and endowed with an algebraic mixed twistor structure. In Section \ref{u1}, this is enhanced to an analytic mixed Hodge structure. These both split on tensoring with the weight $0$ algebra $\cS:=\R[x]$ when the Hodge filtration on $\cS \ten \Cx$ is given by powers of $(x-i)$  (Propositions \ref{formalityuniv} and \ref{formalityuniv2})

We will adopt the terminology and notation of \cite{mhs} without further comment.
%
Throughout this paper, a real analytic stack will be understood to mean an Artin stack on the real analytic site, in the sense of \cite{Noohi1}. 

\section{The de Rham moduli stack}\label{dR}

Fix a compact K\"ahler manifold $X$, and a real pro-algebraic groupoid $R$.

We define the de Rham moduli stack $\fM_{\dR,X,R}$ to be the  real  analytic stack 
parametrising flat connections on principal complex $\C^{\infty}$ $R$-bundles on $X$. Explicitly:

\begin{definition}
Given a sheaf $\sG$ of groupoids on a topological space $X$, define a principal $\sG$-bundle
to consist of a pair $(f,\sP)$, for $f\in\Gamma(X,\Ob\sG)$ a global section of the object set, and $\sP$ a principal $\sG(f,f)$-bundle.  A morphism from $(f,\sP)$ to $(f',\sP')$ consists of  a section $ \alpha\in\Gamma(X,\sG(f,f'))$ together with an isomorphism $\beta:\sP\alpha\to\sP'$   of $\sG(f',f')$-bundles.  Note that if $\sG$ is a group, this recovers the usual definition, and that if $\sG$ is a constant sheaf, then the groupoid of bundles is equivalent, but not isomorphic, to the category of functors from $ \pi_fX  $ to $\sG$.
 \end{definition}

\begin{definition}
 Given a real analytic space  $U$ and a real manifold $Y$, define  the sheaf $\sA^0_{Y\by U/U}$ to consist of $\C^{\infty}$ functions on $Y\by U$, analytic along $U$. It is clear that this is well-defined when $U$ is smooth; if not, take a closed immersion $i:U \into S$, for $S$ smooth,  and set $\sA^0_{Y\by U/U}:=i^{-1} \sA^0_{Y\by S/S} \subset \sA^0_{Y\by U}$. Now set  $\sA^n_{Y\by U/U}:= \sA^0_{Y\by U/U}\ten_{ \sA^0_{Y}}\sA^n_{Y}$.
\end{definition}

\begin{definition}
 Given a 
 real analytic space  $U$, define $\fM_{\dR,X,R}(U)$ to be the groupoid with objects $(\sP,\nabla)$, for
 $\sP$ a principal $R(\sA^0_{X\by U/U}\ten \Cx)$-bundle  on $X$, and 
$$  
\nabla : \sP\to\ad \sP\ten_{\sA^0_{X\by U/U}}\sA^1_{X\by U/U},
$$ 
which is a connection in the sense that $\nabla(p\cdot g)= \ad_g \nabla(p) + g^{-1}\cdot dg$, for $z,z' \in \Ob R, p \in \sP(z), g \in R(\sA^0_{X\by U/U})(z,z')$. We also require that $\nabla$ be flat, i.e. that
 $$
d\circ\nabla =0:\sP \to  \ad \sP\ten_{\sA^0_{X\by U/U}}\sA^2_{X\by U/U}.
 $$

An isomorphism from $(\sP, \nabla)$ to $(\sP', \nabla')$ is an isomorphism $f: \sP \to \sP'$, such that $\nabla \circ f= \ad f \circ \nabla'$. 

Define the involution $\tau$ on  $\fM_{\dR,X,R}$ to be given by complex conjugation of $R(\Cx)$.
\end{definition}

\begin{lemma} \label{anal1}
Suppose that $U,W$ are open balls in $\R^n$ centred at the origin, and $S$  a real analytic subspace of an open disc $V$ in $\R^N$, with all embedded components passing through a point $s \in S$. If $\sP$ is an $R(\sA^0_{U\by W\by S/S})$-torsor on $U\by W \by S$, equipped with a flat connection
$$
\nabla: \sP \to (\ad \sP) \ten_{\sA^0_{U}} \sA^1_U,
$$
then
$$
\nu_S: \H^0(U\by W \by S, \sP)^{\nabla}\to \H^0(\{0\}\by W \by S, \sP|_{\{0\}\by W \by S })
$$
is an isomorphism.
\end{lemma}
\begin{proof}
 If $S$ is smooth, then the result is well known. In particular it is true when $S$ is a point, and hence for an Artinian scheme (replacing $\sP$ by $\pr_{1*}\sP$). The argument of \cite{Sim2} Lemma 7.4 adapts to give injectivity of $\nu$.

For surjectivity, we may assume that $U=U'\by U''$, with $\dim U'=1$. Replacing $U$ with $U'$ and $W$ with $W':=U''\by W$, and arguing inductively, we see that it suffices to consider the case $\dim U=1$. The flatness condition is then vacuous, allowing us to extend $(\sP, \nabla)$ to $(\tilde{\sP}, \tilde{\nabla})$ on $U \by W \by V$ (similarly to \cite{Sim2} Lemma 7.4). Since $\nu_V$ is known to be an isomorphism, this gives the required surjectivity of $\nu_S$.
\end{proof}

\begin{proposition}\label{analytic}
For any contractible 
real analytic space $U$, there is an equivalence of categories between the groupoid of $R_{\Cx}(U)$-representations of $\pi_1(X)$, and $\fM_{\dR,X,R}(U)$. Moreover, $R(U)$-representations  correspond to the $\tau$-invariant locus of $\fM_{\dR,X,R}(U)$.    
\end{proposition}
\begin{proof}
This is really just the observation that $R(U)$-representations of $\pi_1(X)$ correspond to $R(\O_U)$-torsors on $X \by U$. The proof now proceeds as for \cite{Sim2} Theorem 7.1,  substituting Lemma \ref{anal1} for \cite{Sim2} Lemma 7.4.  Given a connection     $(\sP,\nabla)$,  we have an $R(\O_U)$-torsor $\ker \nabla$. To an $R(\O_U)$-torsor $P$, we associate the pair  $(P\by^{R(\O_U)}R(\sA^0_{X\by U/U}),\nabla)$, where  $\nabla$  is given by $  \nabla(p,g)= g^{-1}\cdot dg$, for $z,z' \in \Ob R, p \in P(z), g \in R(\sA^0_{X\by U/U})(z,z')$.
 \end{proof}

\begin{definition}
Given a point $\rho \in \fM_{\dR,X,R}(\R)$, say that $\rho$ is stable if the corresponding representation is pro-reductive, and let the stable de Rham stack $\fM_{\dR,X,R}^s \subset \fM_{\dR,X,R}$ be the (open) substack on stable points.
\end{definition}

\section{The de Rham twistor stack}\label{tdR} 
    
\begin{definition}\label{cdef}
Define $C$ to be the real affine scheme $\prod_{\Cx/\R}\bA^1$ obtained from $\bA^1_{\Cx}$ by restriction of scalars, so for any real algebra $A$, we have $C(A)= \bA^1_{\Cx}(A\ten_{\R}\Cx)\cong A\ten_{\R}\Cx$. We let $C^*$ be the quasi-affine scheme $C - \{0\}$.

Define $S$ to be the real algebraic group  $\prod_{\Cx/\R} \bG_m$ obtained as in \cite{Hodge2} 2.1.2 from $\bG_{m,\Cx}$ by restriction of scalars. Note that there is a canonical inclusion $\bG_m \into S$, and that   $S$ acts on $C$ and $C^*$
by inverse multiplication, i.e.
\begin{eqnarray*}
S \by C &\to& C\\
(\lambda, w) &\mapsto& (\lambda^{-1}w).
\end{eqnarray*}
\end{definition}

Choosing $i \in \Cx$ gives an isomorphism $C \cong \bA^2_{\R}$, giving co-ordinates $u,v$ on $C$ so that the isomorphism $A \by A \cong A\ten_{\R}\Cx$ is written  $(u,v) \mapsto u\ten 1 +v\ten i$. Thus the algebra $O(C)$ associated to $C$ is the polynomial ring $C=\R[u,v]$. $S$ is isomorphic to the scheme $\bA^2_{\bR} -\{(u,v)\,:\, u^2+v^2=0\}$.

\begin{definition}\label{scoords}
Let $C^*(\Cx)$ be the real analytic space of complex-valued points  of $C^*$ (isomorphic to $\Cx\ten_{\R}\Cx -\{0\}$). On $C^*(\Cx)$, we denote the complex co-ordinates $(u,v)$ by $(\alpha, \beta)$, corresponding to $\alpha\ten 1 +\beta\ten i\in \Cx\ten_{\R}\Cx$. 

$C^*$ has an involution $\tau$ given by complex conjugation of the co-ordinates $\Cx$, so $\tau(\alpha, \beta)= (\bar{\alpha},\bar{\beta})$.

There is an isomorphism $C(\Cx)\cong\Cx^2-\{0\}$ given by $(\alpha, \beta)\mapsto (\alpha + i\beta, \alpha - i\beta)$. Under this isomorphism, $S(\Cx) \subset C(\Cx)$ maps to $\Cx^* \by \Cx^*$.
\end{definition}

We will now  define the twistor stack $\fT_{\dR,X,R}$  to be the analytic stack over $C^*(\Cx)$ whose fibre over $(\alpha,\beta)$ classifies flat $\alpha d +\beta d^c$-connections on principal $\C^{\infty}$ $R$-bundles on $X$. The explicit construction follows.
  
\begin{definition}
 Given a 
 real analytic space  $U$ over $C^*(\Cx) $, define $\fT_{\dR,X,R}(U)$ to be the groupoid with objects $(\sP,\nabla)$, for
 $\sP$ a principal $R(\sA^0_{X\by U/U})$-bundle  on $X$, and 
$$  
\nabla : \sP\to\ad \sP\ten_{\sA^0_{X\by U/U}}\sA^1_{X\by U/U},
$$ 
which is an $\alpha d +\beta d^c$-connection in the sense that $\nabla(p\cdot g)= \ad_g \nabla(p) + g^{-1}\cdot (\alpha d +\beta d^c)g$, for $z,z' \in \Ob R, p \in \sP(z), g \in R(\sA^0_{X\by U/U})(z,z')$. We also require that $\nabla$ be flat, i.e. that
 $$
 (\alpha d +\beta d^c)\circ\nabla =0:\sP \to  \ad \sP\ten_{\sA^0_{X\by U/U}}\sA^2_{X\by U/U}.
 $$

An isomorphism from $(\sP, \nabla)$ to $(\sP', \nabla')$ is an isomorphism 
$$
f: \sP\by^{R(\sA^0_{X\by U/U})}R(\sA^0_{X\by U/U}\ten\Cx) \to \sP'\by^{R(\sA^0_{X\by U/U})}R(\sA^0_{X\by U/U}\ten\Cx)
$$
 of $R(\sA^0_{X\by U/U}\ten\Cx)$-bundles, such that $\nabla \circ f= \ad f \circ \nabla'$. 
 
  Define the involution $\tau$ on  $\fT_{\dR,X,R}$ to be given by complex conjugation, lifting $\tau$ on $C^*(\Cx)$.
  \end{definition}
\begin{remarks}
 Note that we may extend this definition to give a stack over $C(\Cx)$, but that this would not be analytic (i.e. not have a presentation), since the fibre over $0$ is too large.
\end{remarks}

Note that $\fM_{\dR,X,R}=\fT_{\dR,X,R}\by_{C^*(\Cx)}(1,0)$. 

\begin{definition}\label{saction}
We may define a  real analytic  $\Cx^*\by \Cx^*\cong S(\Cx)$-action on $\fT_{\dR,X,R}(U)$ over $C^*(\Cx)$ by $(\alpha+i \beta,\alpha -i\beta)\dmd (\sP,\nabla):=(\sP, \alpha\nabla+\beta J\bar{\nabla})$. Here $(\alpha, \beta)$ are the standard co-ordinates on $C$, as in Definition \ref{scoords}. 
This is $\tau$-equivariant in the sense that $\tau((\lambda, \mu)\dmd y)= (\bar{\mu}, \bar{\lambda}) \dmd \tau(y)$, since $\tau(\alpha+i \beta,\alpha -i\beta)= (\bar{\alpha}+i \bar{\beta},\bar{\alpha} +i\bar{\beta})$.
 \end{definition}

\begin{remark}\label{cfdh}
In fact, $\fT_{\dR,X,R}$ is the real analytic stack underlying a complex analytic stack, since the constructions above can all be defined for complex analytic spaces, with $\tau$ becoming an antiholomorphic involution. We have concentrated on the real structure to facilitate comparisons in the rest of the paper. 

There is a canonical embedding $\bG_m \into S$, corresponding on $\Cx$-valued points to the diagonal inclusion $\Cx^* \into \Cx^*\by \Cx^*$ when making use of the isomorphism in Definition \ref{saction}. Since $[C^*/\bG_m] \cong \bP^1$, this gives us a morphism $[\fT_{\dR,X,R}/\Cx^*] \to \bP^1(\Cx)$ of real analytic stacks. If $R = \GL_n$, then the coarse moduli space associated to this stack is precisely the Deligne-Hitchin twistor space, as constructed in \cite{hitchin} and described in \cite{simpsonwgt2} \S 3. 
\end{remark}

\begin{remark}\label{twmhs}
In \cite{mhs}, an algebraic Hodge filtration on an object $Z$ was defined to be an $S$-equivariant extension of $X$ over $C^*$. We may therefore regard $\fT_{\dR,X,R}$ as a kind of Hodge filtration on $\fM_{\dR,X,R}$.
\end{remark}

\section{The pluriharmonic moduli stack}\label{cH}

\begin{definition}
Define an involution $C$ on a real pro-algebraic groupoid $R$ to be a Cartan involution if for all $x,y\in\Ob R$, ${C\tau}$ is is  complex conjugation  with respect to a compact real form of the complex scheme $R_{\Cx}(x,y)$, where $\tau$ denotes complex conjugation. Note that this extends the standard definition (from e.g. \cite{Simpson} \S 4) for pro-algebraic groups.
\end{definition}

\begin{examples}
\begin{enumerate}
\item $\GL_n$ has a Cartan involution, given by $C(A)=(A^{\top})^{-1}$.

\item By \cite{Simpson} Theorem 7, $\varpi_1(X,x)_{\R}^{\red}$ has a Cartan involution, corresponding to $-1 \in U_1$ for the unitary action on Higgs bundles. More generally,
 if   $\rho:\pi_fX \to R(\R)$ is any Zariski-dense morphism, then $R$ is a quotient of $\varpi_f(X)^{\red}$, and $C$ descends to $R$ (by Tannakian duality, since $C(V)\cong V^{\vee}$), so the pair $(R,C)$ satisfies the conditions above.
\end{enumerate}
\end{examples}

From now on, assume that the real pro-algebraic groupoid $R$ is  equipped with a Cartan involution $C$.  

For an  $R(\sA^0\ten \Cx)$-torsor $\sP$, observe that the set of principal  $R_{\Cx}^{C\tau}(\sA^0)$-subbundles $\sQ$ of $\sP$ with $\sP = \sQ\by^{R_{\Cx}^{C\tau}(\sA^0)}R_{\Cx}(\sA^0)$  is isomorphic to the set of global sections $s$ of $\sP/R_{\Cx}^{C\tau}(\sA^0)$. Under this correspondence, $\sQ$ is the inverse image of $s$, while $s$ is the image of $\sQ$. There is then an action of $C\tau$ on $\sP$, given by $C\tau(q, r)= (q, C\tau(r))$, for $r\in R(\sA^0)$. This action is $(R(\sA^0\ten \Cx), C\tau)$-equivariant, in the sense that $C\tau(p \cdot r)= C\tau(p) \cdot C\tau(r)$.

\begin{definition}
Given an  $R(\sA^0\ten \Cx)$-torsor $\sP$ on a  K\"ahler manifold $X$, equipped with a global section $s$ of $\sP/R_{\Cx}^{C\tau}(\sA^0)$, and a flat connection $D: \sP \to (\ad \sP)\ten_{\sA^0}\sA^1$,
define the $d^c$-connection $D^c$ to be  $ J \ad_{C\tau} D: \sP\to (\ad \sP)\ten_{\sA^0}\sA^1$, for $J$ the complex structure.

Explicitly, if we decompose $D$ into antihermitian and hermitian parts (i.e. $\pm 1$-eigenspaces for $\C\tau$)   as $D= d^+ +\vartheta$, so
$$
d^+= \half(D+ C\tau\circ D \circ C\tau), \quad \vartheta=\half(D  - C\tau\circ D \circ C\tau)
$$
(noting that $d^+$ is a connection and $\vartheta$ is $R(\sA^0\ten \Cx)$-equivariant),
then into $(1,0)$ and $(0,1)$ types as $d^+=\pd + \pd^{\dagger}$ and $\vartheta= \theta+\theta^{\dagger}$, then we have
$$
D^c= i\pd -i \pd^{\dagger}  -i\theta +i\theta^{\dagger}.
$$
\end{definition}

\begin{definition}\label{pluri}
Define the pseudocurvature of the section $s$ by $G_K:= [D,D^c]=DD^c+D^cD$. 
The bundle is called harmonic (as in \cite{Simpson} \S 1) if $\Lambda G_K=0$, for $\L$ the formal adjoint to the K\"ahler form $\omega$. The bundle is said to be pluriharmonic (as in \cite{mochi} Remark 2.2) if $G_K=0$.
\end{definition}

\begin{definition}\label{clubdef}
Given $t\in U_1$ (i.e. $t \in \Cx$ and $|t|=1$), and a  flat pluriharmonic connection $D$, define $t\clubsuit D:= d^{+}+t\dmd\vartheta$.
\end{definition}

\begin{lemma}
The $U_1$-action given by $\clubsuit$ preserves the set of flat pluriharmonic  connections, and $\tau(t\clubsuit D)= t\clubsuit (\tau D)$, where $\tau$ denotes complex conjugation.
\end{lemma}
 
 \begin{definition}\label{mhdef}
 Given a 
 real analytic space  $U$, define $\fM_{\cH,X,R}(U)$ to be the groupoid with objects $(\sP,s,\nabla)$, for
 $\sP$ a principal $R(\sA^0_{X\by U/U}\ten \Cx)$-bundle  on $X$,  equipped with a global section $s$ of $\sP/R_{\Cx}^{C\tau}(\sA^0)$, and 
$$  
\nabla : \sP\to\ad \sP\ten_{\sA^0_{X\by U/U}}\sA^1_{X\by U/U}
$$ 
 a flat pluriharmonic connection in the sense of Definition \ref{pluri}. 

An isomorphism from $(\sP,s, \nabla)$ to $(\sP',s', \nabla')$ is an isomorphism $f: \sP \to \sP'$, such that $f(s)=s'$ and $\nabla \circ f= \ad f \circ \nabla'$. 
\end{definition}
     
  \begin{lemma}
  There is a canonical map $\Upsilon:\fM_{\cH,X,R}\to \fM_{\dR,X,R}$, and this maps  isomorphism classes of $\R$- valued  points isomorphically to the reductive representations.
  \end{lemma}   
 \begin{proof}
This comes from forgetting the section $s$. Note that the image of this map is      characterised  in \cite{Simpson} Theorem 1 as the semisimple local systems when $(R, C)=(\GL_n, \text{transpose})$. This is because a harmonic metric on a local system $\vv$ defines a $\C^{\infty}$ $U_n$-torsor. By Tannakian duality, it follows that the image in general comprises the reductive representations. 

Moreover, it follows from \cite{mochi} Theorem 1.1 that the harmonic metric on $\vv$ is unique up to conjugation by $\Aut_X(\vv)$ (by uniqueness up to a scalar for irreducible local systems). This shows that $\Upsilon$ is injective on isomorphism classes when $R=\GL_n$, and hence in general. 
  \end{proof}
\begin{remark}
Note that the functor $\Upsilon$ is not full on $\R$-valued points: if $X$ is a point, then $ \fM_{\cH,X,R} \simeq R^{C\tau}$, while $\fM_{\dR,X,R} \simeq R$.
\end{remark}
  
 \section{The pluriharmonic twistor stack}\label{tcH}
 
 \begin{definition}
Define $Q$ to be the real scheme  $Q(A)=\bH\ten A$  representing the quaternions, so $Q\cong\bA_{\R}^4$. We have a real analytic isomorphism $\jmath: C(\Cx) \to Q(\R)$, given 
 by mapping $(\alpha,\beta)$  to $\alpha+\beta j$. Let $Q^*=Q-\{0\}$. 
 \end{definition} 

\begin{definition}\label{relpluristack}
 Given a 
 real analytic space  $U$ over $Q^*$, define $\fT_{\cH,X,R}(U)$ to be the groupoid with objects    $(\sP,s,\nabla)$, for
 $\sP$ a principal $R(\sA^0_{X\by U/U})$-bundle  on $X$,  equipped with a global section $s$ of $\sP/R^C(\sA^0)$, and 
$$  
\nabla : \sP\to\ad \sP\ten_{\sA^0_{X\by U/U}}\sA^1_{X\by U/U}\ten\Cx
$$ 
an $\alpha d +\beta d^c$-connection. 
 We also require that $\nabla$ be flat and pluriharmonic, i.e. that
 $$
 (\alpha d +\beta d^c)\circ\nabla =0:\sP \to  \ad \sP\ten_{\sA^0_{X\by U/U}}\sA^2_{X\by U/U},
 $$
 and that the commutator $[\nabla,\nabla^c]:=\nabla\nabla^c+\nabla^c\nabla$ vanishes, for $\nabla^c= J \ad_{C\tau} \nabla$.

An isomorphism from $(\sP,s, \nabla)$ to $(\sP',s', \nabla')$ is an isomorphism 
$$
f: \sP\by^{R(\sA^0_{X\by U/U})}R(\sA^0_{X\by U/U}\ten\Cx) \to \sP'\by^{R(\sA^0_{X\by U/U})}R(\sA^0_{X\by U/U}\ten\Cx)
$$
 of $R(\sA^0_{X\by U/U}\ten\Cx)$-bundles, such that $\nabla \circ f= \ad f \circ \nabla'$ and $\nabla \circ f= \ad f \circ \nabla'$. 
  \end{definition}

\begin{remarks}
 Note that we may extend this definition to give a stack over $Q$, but that this would not be analytic (i.e. not have a presentation), since the fibre over $0$ is too large.
\end{remarks}

Observe that $\fM_{\cH,X,R}=\fT_{\cH,X,R}\by_{Q^*}(1,0)$. 

\begin{definition}
  Define the involution $\tau$ on  $\fT_{\cH,X,R}$ to be given by complex conjugation, lifting the map $\tau=\ad_j$ on $Q^*$ given by $\alpha+\beta j\mapsto \bar{\alpha}+\bar{\beta} j$.
  \end{definition}

\begin{definition}\label{stardef}
We may define a real analytic  $Q^*$-action on $\fT_{\cH,X,R}(U)$ over $Q^*$ by $(\alpha+\beta j)\star (\sP,\nabla):=(\sP,\alpha\nabla+\beta J\ad_{C\tau}\nabla)$. This is $\tau$-equivariant in the sense that $\tau((\alpha+\beta j)\star y)= (\bar{\alpha}+\bar{\beta} j)\star \tau(y)$.
\end{definition}

\begin{remark}
Observe that,  although multiplication on the scheme $Q^*$ does not have an inverse, its $\R$-valued points have a group structure, since they coincide with $\R$-valued points of the open subscheme $|\alpha|^2 +|\beta|^2\ne 0$.
 \end{remark}

\begin{proposition}
$\fT_{\cH,X,R} \cong Q^*\by\fM_{\cH,X,R}$ as a real analytic stack over $Q^*$.
\end{proposition}
\begin{proof}
The isomorphism is given by $z \mapsto (\pi(z), \pi(z)^{-1}\star z)$, for $\pi: \fT_{\cH,X,R} \to Q^*$. The inverse is $(q, y) \mapsto q \star y$.
\end{proof}

 \begin{definition}\label{upsdef}
  Define $\Upsilon:\fT_{\cH,X,R}\to \fT_{\dR,X,R}$ by $(\sP,s,\nabla)\mapsto (\sP,\nabla)$; this is a morphism over $\jmath^{-1}:  Q^*(\R) \to C^*(\Cx)$.
  \end{definition}   

\begin{definition}
Let $j:= 1\ten i \in \Cx\ten_{\R}\Cx$.  For $\gamma \in C(\Cx)$, this gives $\gamma = \alpha +\beta j$ in terms of the standard co-ordinates.
Note that $j$ commutes with $\Cx \ten 1$, and that the isomorphism $\jmath$ is then given by $\alpha +\beta j \mapsto \alpha +\beta j$.
\end{definition}

 Observe that $S(\Cx)$ then consists of $\alpha +\beta j$ for which $\alpha^2+\beta^2$ is invertible, and that on a $(p,q)$-form, the $\dmd$-action of $\alpha +\beta j \in S(\Cx)$ is $(\alpha +i\beta)^p(\alpha -i\beta)^q$.

\begin{proposition}\label{upsprops}
For $(\alpha +\beta j ) \in  S(\Cx)$, such that $|\alpha +i\beta|=|\alpha -i\beta |$, and $y$ in $\fM_{\cH,X,R}$
$$
(\alpha +\beta j ) \dmd \Upsilon(   y) = \Upsilon ((\alpha +\beta j ) \star (\frac{\alpha +i\beta }{\alpha -i\beta} \clubsuit y)),
$$
for $\clubsuit$ as in Definition \ref{clubdef}.
In particular, for $\beta=0$ (corresponding to $\bG_m(\Cx) \subset S(\Cx)$), the $\star$ and $\dmd$-actions agree.
\end{proposition}
\begin{proof}
 This follows from the observation  that for $D$ a $d$-connection (associated to an object of $\fM_{\cH,X,R} $),
\begin{eqnarray*}
(\alpha +\beta j) \star D&=& (\alpha +\beta J) d^+ + (\alpha - \beta J) \vartheta\\
&=& (\alpha +i\beta)(\pd  + \theta^{\dagger}) +(\alpha -i\beta)(\pd^{\dagger}  +\theta)\\
&=& (\alpha +\beta j )\dmd(d^+ + ( \frac{\alpha^2+\beta^2 +2\alpha\beta j}{\alpha^2 +\beta^2})\dmd\vartheta)
.
\end{eqnarray*}
Thus 
$$
(\alpha +\beta j )\dmd \Upsilon( \frac{\alpha -i\beta}{\alpha +i\beta }  \clubsuit  z) = \Upsilon (  (\alpha +\beta j )\star z),
$$
and the result follows by setting $z= \frac{\alpha +i\beta }{\alpha -i\beta} \clubsuit y$.
\end{proof}

\begin{remark}\label{involutions}
Making use of the isomorphism $\bP^1_{\Cx}\cong [C^*_{\Cx}/\bG_{m,\Cx}]$, and the compatibility of the $\Cx^*$ actions for $\star$ and $\dmd$, we  obtain an analytic morphism 
$$
\Upsilon: [\fM_{\cH,X,R}/\bG_m(\Cx)] \to [\fM_{\dR,X,R}/\bG_m(\Cx)]
$$
over $\bP^1(\Cx)= [C^*(\Cx)/\bG(\Cx)]$. 

$\Upsilon$  is $\tau$-equivariant, where $\tau$ is the involution $\tau(\alpha:\beta)=(\bar{\alpha}:\bar{\beta})$. Making use of the co-ordinate change $(\lambda: \mu)=(\alpha+i \beta: \alpha -i\beta)$ as Definition \ref{scoords}, $\tau$  corresponds to the circular involution $(\lambda: \mu) \mapsto (\bar{\mu}: \bar{\lambda})$ of \cite{MTS} p.12 . The antipodal involution $\sigma(\lambda: \mu)= (i\bar{\lambda}, -i\bar{\mu})$ of \cite{MTS} corresponds in our co-ordinates to $ \sigma(\alpha:\beta)= (-\bar{\beta}: \bar{\alpha})$ --- this lifts to $[\fM_{\cH,X,R}/\Cx^*]$ as $j\star$. 
\end{remark}

\begin{remark}\label{cfsimpson}
We may also use these actions to describe the discrete $\Cx^*$-action given in \cite{Simpson} on the points  of $\fM_{\cH,X,R}$. $\lambda\in \Cx^*$ maps $y$ to
$$
(1-ij)\star(\Upsilon^{-1}(\frac{\lambda+1}{2} +\frac{\lambda-1}{2i} j)\dmd\Upsilon(\frac{1+ij}{2}\star y)),
$$
making use of the injectivity of $\Upsilon$ on isomorphism classes of objects, and the fact that $\im \Upsilon$ is preserved by the $\dmd$-action. 

This maps $D=\pd + \pd^{\dagger}+\theta+\theta^{\dagger}$ to a connection $D'$ with $(\pd^{\dagger})'=\pd^{\dagger}$ and $\theta'=\lambda \theta$. If $\lambda \in U_1 \subset \Cx^*$, this expression reduces to $\lambda \clubsuit D$.
\end{remark}

\section{Local structure of the moduli stacks}\label{local}

Fix an $\R$-valued object $y$ of $\fT_{\cH,X,R}$, i.e. a point $\alpha +\beta j \in \bH^*$, together with a triple $(\sP, s, \nabla)$ as in Definition \ref{relpluristack}, for $\nabla$ a flat pluriharmonic $\alpha d +\beta d^c$-connection. We now describe the local structure of $\fT_{\cH,X,R}$ over $Q^*$ about $y$, and of  $\fT_{\dR,X,R}$ over $C^*(\Cx)$ about $\Upsilon(y)$. 

As in \cite{GM}, \S 3.1, the analytic germs are determined by evaluating the stack at Artinian local $\R$-algebras. Explicitly, let $\C_{\R}$ be the category of  Artinian local $\R$-algebras with residue field $\R$. Then the germ of $\fT_{\cH,X,R}$ at $y$ is determined by the groupoid-valued functor 
\begin{eqnarray*}
(\fT_{\cH,X,R})_y: \C_{\R} &\to& \Gpd \\
A & \mapsto& \fT_{\cH,X,R}(\Spec A)\by_{ \fT_{\cH,X,R}(\Spec \R)} y.
\end{eqnarray*}

We may also describe $Q^*$ in this way, with 
\begin{eqnarray*}
Q^*_{\alpha +\beta j}: \C_{\R} &\to& \Set \\
A & \mapsto& Q^*(\Spec A)\by_{ Q^(\Spec \R)}(\alpha +\beta j);
\end{eqnarray*}
this is pro-represented by
$\widehat{O(Q)_{\alpha +\beta j}}$, the complete local ring of $Q$ at $\alpha +\beta j$, which is isomorphic to the real power series ring in $4$ variables. 

There is a similar description for $\fT_{\dR,X,R}$ (noting that $C^*(\Cx)$ is a smooth $4$-dimensional real analytic space).

In \cite{GM} \S 2, a deformation groupoid is associated to every real differential graded Lie algebra (DGLA) $L$.
\begin{definition} Fix a real DGLA $L$. The Maurer-Cartan functor 
$\mcl:\C_{\R} \ra \Set$ is defined by 
$$ 
\mcl(A)=\{x \in L^1\otimes \m(A) |dx+\half[x,x]=0\}.
$$

Define the gauge functor $G_L:\C_{\R} \ra \Grp$ by $G_L(A)=\exp(L^0 \ten \m(A))$. This acts on $\mcl$ by the gauge action $g(x):= gxg^{-1} -(dg)g^{-1}$.
Define the deformation functor $\Def_L:\C_{\R} \to \gpd$ by $\Def_L:= [\mcl/G_L]$, and say that $L$ governs a functor $F$ if $F$ is equivalent to $\Def_L$.
\end{definition}

\begin{definition}
Given $y= (\alpha + \beta j,\sP, s, \nabla)\in (\fT_{\cH,X,R})(\R)$, define
\begin{eqnarray*}
A^{\bt}(X,\ad y)&:=& \H^0(X, (\ad \sP\ten_{\sA^0_{X\by U/U}}\sA^*_{X\by U/U}, \nabla))\\
\H^*(X,\ad y)&:=& \H^*( A^{\bt}(X,\ad y)).
\end{eqnarray*}
\end{definition}

\begin{lemma}\label{gmtwist}
Given $y= (\alpha + \beta j,\sP, s, \nabla)\in (\fT_{\cH,X,R})(\R)$,
the functor $(\fT_{\dR,X,R})_{\Upsilon(y)}: \C_{\R} \to \Gpd$ is governed by the DGLA
$$
L^i:= \left\{\begin{matrix} A^i(X,\ad y) & i \ne 1\\
\Cx\ten_{\R}\Cx \oplus A^1(X,\ad y) & i =1, \end{matrix} \right.
 $$
with the standard differential and Lie bracket on $A^{\bt}(X,\ad y) $, while $d(\Cx\ten_{\R}\Cx )=0$, and  $[u+vj, w]= \frac{(u\bar{\alpha} - v\bar{\beta})\nabla w + (v\alpha -u\beta)\nabla^c w}{|\alpha|^2+|\beta|^2}$, for  $u+vj \in \Cx\ten_{\R}\Cx$ and $w \in  A^i(X,\ad y)$.
\end{lemma}
\begin{proof}
Fix $A \in \C_{\R}$, and denote the maximal ideal by $\m(A)$. Since $\C^{\infty}$-bundles do not deform, any $A$-valued deformation of $(\sP, \nabla)$ is isomorphic to a deformation of the form $(\tilde{\alpha}+\tilde{\beta}j,\sP\by^RR(A), \tilde{\nabla})$. Here, $\tilde{\alpha}+ \tilde{\beta}j \in A\ten_{\R} \Cx\ten_\R\Cx$, congruent  to $\alpha + \beta j\mod \m(A)$. Similarly, 
$$
\tilde{\nabla} : \sP\to\ad \sP\ten_{\sA^0_{X\by U/U}}\sA^1_{X\by U/U}\ten A,
$$  
is a flat $\tilde{\alpha} d +\tilde{\beta} d^c$-connection, congruent to $\nabla \mod \m(A)$.

Now, given an element $(u+vj, \omega) \in \mc_L(A)\subset L^1\ten \m(A)$ of the Maurer-Cartan space, define
\begin{eqnarray*}
(\tilde{\alpha}+ \tilde{\beta}j)&:=& (\alpha + \beta j)+(u+vj)\\
\tilde{\nabla}&:=& \jmath(\tilde{\alpha}+ \tilde{\beta}j)\star \jmath(\alpha + \beta j)^{-1}\star \nabla +\omega.
\end{eqnarray*}

This gives the required correspondence, noting that flatness of $\tilde{\nabla}$ is equivalent to the Maurer-Cartan equation
$$
d(u+vj, \omega)+\half[(u+vj, \omega), (u+vj, \omega)]=0.
$$
The treatment of isomorphisms is similar (as in \cite{GM} \S 6).
\end{proof}

The following result generalises the formality results of \cite{GM}, and shows that the moduli stack has only quadratic singularities in the image of $\Upsilon$.
\begin{proposition}\label{cfgm}
Given $y= (\alpha + \beta j,\sP, s, \nabla)\in (\fT_{\dR,X,R})(\R)$, then  the germ $(\fT_{\dR,X,R}, \Upsilon(y))$ is isomorphic to $(C^*(\Cx), \alpha +\beta j) \by \Def_{(\H^*(X,\ad y))}$ over $C^*(\Cx)$.
\end{proposition}
\begin{proof}
By Lemma \ref{gmtwist}, the germ is given by $\Def_L$. Now, we have quasi-isomorphisms 
$$
\H^*(L) \la (\ker \nabla^c, \nabla) \to L
$$
of DGLAs by the principle of two types (similarly to \cite{mhs} Lemma \ref{mhs-formalitysl}), so $\Def_L \cong \Def_{\H*(L)}$. The quasi-isomorphisms also respect the augmentation maps from these DGLAs to $ \Cx\ten_{\R}\Cx[-1]$, the DGLA governing $(C^*(\Cx), \alpha +\beta j)$. Now, $\H^*(L)\cong \frac{\ker \nabla \cap \ker \nabla^c}{\im \nabla\nabla^c}$, so $\H^*(L)= \H^*(X,\ad y) \by \Cx\ten_{\R}\Cx[-1]$.
\end{proof}

\begin{definition}
Given $a \in A^*(X,\ad y)$, set $a^c:= JC\tau a$. 
\end{definition}

Observe that the description $\H^*(X, \ad y)\cong \frac{\ker \nabla \cap \ker \nabla^c}{\im \nabla\nabla^c}$ ensures that $JC\tau$ is a well-defined automorphism of $\H^*(X, \ad y)$.

\begin{proposition}\label{localpluri}
The isomorphism classes of the germ $(\fM_{\cH,X,R}, y)$ have tangent space $\H^1(X,\ad y)$ and obstruction space $\H^2(X, \ad y) \oplus \H^2(X, \ad y)^{-JC\tau}$, while the automorphism group of $y$ has  tangent space $\H^0(X,\ad y)^{C\tau}$. The action of the automorphism group on the tangent and obstruction spaces  comes from exponentiating the Lie bracket, and the primary obstruction map is
\begin{eqnarray*}
\H^1(X,\ad y) &\to& \H^2(X, \ad y) \oplus \H^2(X, \ad y)^{-JC\tau}\\
\omega &\mapsto& ([\omega,\omega], [\omega, \omega^c]).
\end{eqnarray*}
\end{proposition}
\begin{proof}
Observe that a flat $d$-connection $\nabla$ is pluriharmonic if and only if $\nabla +\nabla^c$ is a flat $d+\dc$-connection. 
Thus  $(\fM_{\cH,X,R}, y)$ is governed by the DGLA
$M \subset A^{\bt}(X,\ad y)\by (A^*(X,\ad y), \nabla +\nabla^c)$, given by 
$$
M^i:= \left\{\begin{matrix} 
\{(a,a) \,:\, a \in A^0(X,\ad y)^{C\tau}\} & i =0,\\
\{(a, a+a^c)  \,:\, a \in A^1(X,\ad y)\} & i =1,\\
  \{(a,b) \,:\,b -a-a^c \in A^2(X,\ad y)^{-JC\tau}\} &i=2,\\ 
A^i(X,\ad y)\by A^i(X,\ad y) & i >2. \end{matrix} \right.
$$
This follows because $A^{\bt}(X,\ad y)\by (A^*(X,\ad y), \nabla +\nabla^c)$ governs deformations $(\nabla', E')$, of the pair $(\nabla,\nabla +\nabla^c )$, and $M$ imposes the additional constraint that $E'=\nabla' +(\nabla')^c$.

We now calculate the cohomology groups of this DGLA. Observe that $H^0(M)=\H^0(X,\ad y)^{C\tau}$. For $\H^1$, take $(\omega, \omega+\omega^c) \in \z^1(M)$, so $(\nabla \omega, \nabla \omega + \nabla^c\omega+\nabla \omega^c +\nabla^c\omega^c)=0$, or equivalently $\nabla \omega=  \nabla^c\omega+\nabla \omega^c =0$. By the principle of two types, we may write $\nabla^c\omega= \nabla^c\nabla\alpha$, for some $\alpha \in A^0(X,\ad y)$. Applying $JC\tau$, we also have $\nabla\omega^c= \nabla\nabla^c\alpha^c$, so $(\nabla^c\alpha, \nabla^c\alpha+ (\nabla^c\alpha)^c) \in  \z^1(M)$. Decompose $\alpha$ as $\alpha=\alpha^+ +\alpha^-$ into $\pm 1$-eigenvectors for $C\tau$, and observe that $(\nabla\alpha)^c= \nabla^c\alpha^+-\nabla^c\alpha^-$. Thus $0=\nabla(\nabla\alpha)^c+\nabla^c(\nabla\alpha)=2\nabla^c\nabla\alpha^-$, which means that $\nabla\alpha^- \in \im(\nabla\nabla^c)=0$, by the principle of two types, so we may assume that $\alpha^-=0$. Since $\omega - \nabla\alpha \in \ker(\nabla) \cap \ker(\nabla^c)$, we have
$$
\z^1(M)= \{(a, a+a^c)\,:\, a \in (\ker(\nabla) \cap \ker(\nabla^c))\oplus \nabla(M^0)\},
$$
so $\H^1(M) \cong \ker(\nabla) \cap \ker(\nabla^c)\cong \H^1(X,\ad y)$.

To calculate $\H^2$, take $z \in \H^2(M)$, represented by  $(a,b) \in \z^2(M)$. We may choose $e \in A^1(X,\ad y)$ such that $a + \nabla e \in \ker  \nabla^c$, so without loss of generality, we may assume that $a \in \ker \nabla \cap \ker \nabla^c$. Then 
$$
\nabla^c(b-a-a^c)= (-\nabla(b-a-a^c)^c)^c= (\nabla(b-a-a^c))^c=0,
$$
so $b \in \ker \nabla \cap \ker \nabla^c$.
By the principle of two types, any other such representative is of the form $(a+ \nabla e, b+ (\nabla+\nabla^c)  (e+e^c))$, for $e \in \ker \nabla^c\nabla$. Thus:
$$
\H^2(M)\cong \frac{\{(a,b) \in  (\ker \nabla \cap \ker \nabla^c)\by \ker (\nabla +\nabla^c)\,:\, b -a-a^c \in A^2(X,\ad y)^{-JC\tau}\} }{\{  \nabla e,(\nabla+\nabla^c)  (e+e^c)) \,:\,e \in \ker \nabla^c\nabla  }.
$$
If we now change co-ordinates on $A^*(X, \ad y)^2$, by the transformation $(a,b) \mapsto (a, b-a-a^c)$, then in these new co-ordinates we have
$$
\H^2(M)\cong \frac{(\ker \nabla \cap \ker \nabla^c)\by \ker (\nabla +\nabla^c)^{-JC\tau}}{\{  \nabla e, \nabla e^c +\nabla^c e) \,:\,e \in \ker \nabla^c\nabla \} }
$$
Since $\ker \nabla^c\nabla=\ker \nabla^c +\ker \nabla$, we may write $e=f+g$ with respect to this decomposition, giving $(a+ \nabla f, \beta+ \nabla g^c +\nabla^c g)$. Now $\nabla (\ker \nabla^c)= \nabla^c(\ker \nabla)= \nabla^c\nabla A^0(X,\ad y)$, so the denominator is
$$
\im (\nabla^c\nabla )\by \{ v- v^c\,:\, v \in \im (\nabla^c\nabla )\}= \im (\nabla^c\nabla )\by \im (\nabla^c\nabla )^{-JC\tau}.
$$
Hence
$$
\H^2(M)\cong \frac{(\ker \nabla \cap \ker \nabla^c)\by \ker (\nabla +\nabla^c)^{-JC\tau}}{ \im (\nabla^c\nabla )\by \im (\nabla^c\nabla )^{-JC\tau} }\cong\H^2(X, \ad y) \oplus \H^2(X, \ad y)^{-JC\tau}.
$$

The remaining statements follow by computing the Lie bracket on $\H^*(M)$.
\end{proof}

\section{Relative homotopy types over analytic moduli stacks}\label{relhtpy}

The following is a partial generalisation of \cite{mhs} Definition \ref{mhs-hoc*} to arbitrary ringed topoi:
\begin{definition}\label{ringtop}
Given a ringed topos $Y$,  define  $DG_{\Z}\Alg_{Y}(R)$ to be the category of $R$-representations in  quasi-coherent $\Z$-graded cochain algebras on $Y$. Define a weak equivalence in this category to be a map giving isomorphisms on cohomology sheaves (over $Y$), and define  $\Ho(DG_{\Z}\Alg_{Y}(R))$ to be the homotopy category obtained by localising at weak equivalences. Define the categories $dg_{\Z}\Aff_Y(R), \Ho(dg_{\Z}\Aff_Y(R))$ to be the opposite categories.
\end{definition}

\begin{remark}
If, for the $R$-representation $Y$ in schemes of \cite{mhs} Definition \ref{mhs-hoc*},  we let $Z$ be the algebraic stack 
$$
[\xymatrix@1{ \coprod_{x,x' \in \Ob R} R(x,x') \by Y(x') \ar@<1ex>[r] \ar@<-1ex>[r] & \ar[l] \coprod_{x \in \Ob R} Y(x)}]
$$
in the notation of \cite{Champs} 2.4.3, then the categories $DG_{\Z}\Alg_{Z}$ and $DG_{\Z}\Alg_{Y}(R)$ are equivalent, and similarly for the other constructions in Definition \ref{ringtop}

Similarly, we have $DG_{\Z}\Alg_{Y}(R) \simeq DG_{\Z}\Alg_{Y\by R}$ for the site $Y$ of Definition \ref{ringtop}, regarding $R$ as an algebraic stack (with the lisse-\'etale site).
\end{remark}

\begin{definition}\label{xunivdef}
Define $X^{R, \univ} \in \Ho(dg_{\Z}\Aff_{[\fM_{\cH,X,R}/\langle \tau \rangle]}(R))$ to be given by the flat object $\Spec A^{\bt}(X,O)$, defined as follows.  For $ f:V \to \fM_{\cH,X,R}$  given by    $(\sP,s,\nabla)$ as in Definition \ref{mhdef}, we  set $\sA^n(\sP):= \sP\by^{R(\sA^0_{X\by V/V}\ten \Cx)}(\sA^n_{X\by V/V}\ten \Cx\ten O(R))$. 
Now, for $V$ contractible define 
$$
\Gamma({V }, A^{*}(X, O)):= \Gamma(X, \sA^*(\sP)),
$$
with differential $\nabla$. Note that this definition does not involve $s$, so $X^{R, \univ}$ is the pullback of an object on $\fM_{\dR,X,R}$.
\end{definition}

\begin{theorem}\label{univmts}
Given a compact K\"ahler manifold $X$, and a real pro-algebraic groupoid $R$ with a Cartan involution $C$,  
there is a canonical object 
$$
X_{\MTS}^{R, \univ} \in \Ho(dg_{\Z}\Aff_{[\fM_{\cH,X,R}/\langle \tau \rangle] \by  [\bA^1/\bG_m]\by [C^*/\bG_m]}(R)),
$$ 
which on pulling back along the point $[\rho]: \Spec \R \to [\fM_{\cH,X,R}/\langle \tau \rangle] $ (corresponding to a real Zariski-dense representation $\rho: \pi_f \to R(\R)$) gives the mixed twistor structure  
 $X_{\MTS}^{\rho, \mal} \in \Ho(dg_{\Z}\Aff_{ [\bA^1/\bG_m]\by[C^*/\bG_m]}(R))$ of \cite{mhs} Theorem \ref{mhs-mtsmal}. 

  Moreover, $X_{\MTS}^{R, \univ}\by_{[\bA^1/\bG_m]\by[C^*/\bG_m], (1,1)}^{\oR}\Spec \R \cong X^{R, \univ}$.
\end{theorem}
\begin{proof}
We begin by defining the object
$$
X_{\bT}^{R, \univ} =X_{\MTS}^{R, \univ}\by_{[\bA^1/\bG_m], 1}^{\oR}\Spec \R \in \Ho(dg_{\Z}\Aff_{[\fM_{\cH,X,R}/\langle \tau \rangle] \by [C^*/\bG_m]}(R)).
$$ 
Since $\fS:=[\fM_{\cH,X,R}/\langle \tau \rangle] \by [C^*/\bG_m]$ is a product of an analytic stack with an algebraic stack, a little care has to be taken with this site. A base for the site is given by objects of the form $[V/ \langle \tau \rangle] \by [W/\bG_m ]$, for $f:V \to \fM_{\cH,X,R}$ a $\tau$-equivariant smooth map of analytic stacks (i.e. a smooth map $[V/ \langle \tau \rangle] \to\fM_{\cH,X,R}/\langle \tau \rangle $ over $B\langle \tau \rangle$), and $g: W \to  C^*$ a smooth $ \bG_m$-equivariant map of algebraic spaces. The structure sheaf $\O_{\fS}$ on this site is given by $\O_{\fS}|_{V\by W}= (\pr_V^{-1} \O_V)\ten_{\R} (\pr_W^{-1} \O_W)$. 

Let $ f:V \to \fM_{\cH,X,R}$ be given by    $(\sP,s,\nabla)$ as in Definition \ref{mhdef}, and take $g:W \to C^*$, $\bG_m$-equivariant. 
Now, for $V$ contractible and $W$ affine, define the sheaf $\tilde{A}^n(X,O)$ locally by
$$
\Gamma({V \by W}, \tilde{A}^{n}(X, O)):= \Gamma(X, \sA^n(\sP))\ten_{\R} \Gamma(W, \O_W).
$$
 This has a $\tau$-action given by combining the $\tau$-action on $\sP$ (or equivalently, on $V$) with complex conjugation on $\Cx$. It also has a $\bG_m$-action, given by combining the $\bG_m$-action on $W$ with the action setting $\sA^n(\sP)$ to be pure of weight $n$. We therefore define $\tilde{A}^n(X,O)$ on $[V/ \langle \tau \rangle] \by [W/\bG_m]$ to be given by $\langle \tau \rangle \by\bG_m$-invariants: 
$$
\Gamma( [V/ \langle \tau \rangle] \by [W/\bG_m], \tilde{A}^{n}(X, O)):= \Gamma(V \by W, \tilde{A}^{n}(X, O))^{ \langle \tau \rangle, \bG_m}.
$$

To see that this is  a quasi-coherent sheaf, first consider the moduli stack $\fM_{\C^{\infty},X,R}$ of complex principal $\C^{\infty}$ $R$-bundles on $X$. This is a discrete stack, so provided $V$ is connected, we may assume that there is a principal $R(\sA^0_X\ten \Cx)$-bundle $\sB$, with $\sP= \sB\by^{R(\sA^0_X\ten \Cx)}R(\sA^0_{X\by V/V}\ten \Cx)$. Then 
$$
\sA^n(\sP)= \sB\by^{R(\sA^0_X\ten \Cx)}(\sA^n_{X\by V/V}\ten \Cx\ten O(R)),
$$
and compactness of $X$ then gives
$$
\Gamma(X, \sA^n(\sP))= \Gamma(X, \sB\by^{R(\sA^0_X\ten \Cx)}(\sA^n_X\ten \Cx\ten O(R) ))\ten \Gamma(V, \O_V),
$$
since we may regard $\sA^n_{X\by V/V}$ as consisting of $\C^{\infty}$ functions from $X$ to $\O_V$. This gives
$$
\tilde{A}^{n}(X, O) =\Gamma(X, \sB\by^{R(\sA^0_X\ten \Cx)}(\sA^n_X\ten \Cx\ten O(R) ))\ten \O_{\fS}.
$$

We now define the differential on $\tilde{A}^*(X, O)$. The co-ordinates $u,v$ on $C^*$ give elements of $\sO_W$, and we now set the differential on  $\Gamma(V \by W, \tilde{A}^*(X, O))$ to be $u\nabla +v\nabla^c$. Since this is $\langle \tau \rangle \by\bG_m$-equivariant, it descends  to $\Gamma( [V/ \langle \tau \rangle] \by [W/\bG_m], \tilde{A}^{n}(X, O))$, so we have defined
$$
\tilde{A}^{\bt}(X, O) \in DG_{\Z}\Alg_{[\fM_{\cH,X,R}/\langle \tau \rangle] \by [C^*/\bG_m]}(R),
$$
and may set $X_{\bT}^{R, \univ}:= \Spec \tilde{A}^{\bt}(X, O)$.

We now define the weight filtration $W$ on $\tilde{A}^{\bt}(X, O)$ to be given by good truncation $W_r = \tau_{\le r}$, and we define $O(X_{\MTS}^{R, \univ})$ to be the Rees algebra $\Rees(\tilde{A}^{\bt}(X, O), W)$ of this filtration. Since this complex is bounded and flat, derived pullbacks agree with ordinary pullbacks, so
$$
\Spec \Rees(\tilde{A}^{\bt}(X, O), W)\by_{C^*,1}^{\oR} \Spec \R= \Spec\Rees(A^{\bt}(X, O), W), 
$$
as required.
\end{proof}

\subsection{Formality and splitting}

\begin{lemma}\label{typesuniv}
On the $[\fM_{\cH,X,R}/\langle \tau \rangle]$-sheaf $A^{*}(X,O)$ of Definition \ref{xunivdef},  the operators $\nabla, \nabla^c$ satisfy the principle of two types
$$
\ker \nabla \cap \ker \nabla^c \cap (\im \nabla + \im \nabla^c)=\im \nabla\nabla^c.
$$
\end{lemma}
\begin{proof}
It suffices to show that for any real analytic space $U$, and any $U$-valued family $(\sV, \langle \rangle, \nabla)$ of pluriharmonic vector bundles on $X$, that the  operators $\nabla, \nabla^c$ satisfy the principle of two types on $A^*(X,\sV)$. The proof of the Hodge decomposition on forms adapts to this generality, observing that the Green's operator preserves real analytic families of forms.
\end{proof}

\begin{definition}
Define $\ugr X_{\MTS}^{R, \univ} \in \Ho(dg_{\Z}\Aff_{[\fM_{\cH,X,R}/\langle \tau \rangle] \by   B\bG_m}(R))$ by setting
$$
O(\ugr X_{\MTS}^{R, \univ}):= \H^*(A^{\bt}(X, O) ).
$$
\end{definition}

\begin{definition}\label{rowdef}
Define an $S$-action on  $\SL_2$ to be  given on the top row as right multiplication by 
$\lambda \mapsto \left(
\begin{smallmatrix} \Re  \lambda & \Im\lambda  \\ -\Im \lambda & \Re \lambda \end{smallmatrix} \right)^{-1},$ and on the bottom row by 
$\lambda \mapsto \left(
\begin{smallmatrix} \Re  \lambda & -\Im\lambda  \\ \Im \lambda & \Re \lambda \end{smallmatrix} \right).$

Let $\row_1 :\SL_2 \to C^*$ be the $S$-equivariant map given by projection onto the first row.  Under the equivalence between $S$-equivariant $C^*$-bundles and Hodge filtrations of \cite{mhs} Corollary \ref{mhs-flathfil}, $\row_1$  corresponds to the algebra
$$\cS:=\R[x],$$
with  filtration $F^p (\cS\ten \Cx)= (x-i)^p\Cx[x]$, by \cite{mhs} Lemma \ref{mhs-slhodge}.
\end{definition}

\begin{proposition}\label{formalityuniv}
There is a canonical isomorphism
$$
(X_{\MTS}^{R, \univ})\by_{[C^*/\bG_m], \row_1}^{\oR}[\SL_2/\bG_m] \simeq  \ugr X_{\MTS}^{R, \univ} \by_{B\bG_m}^{\oR} ([\bA^1/\bG_m] \by[\SL_2/\bG_m]),
$$
in $\Ho(dg_{\Z}\Aff_{[\fM_{\cH,X,R}/\langle \tau \rangle] \by   [\bA^1/\bG_m]\by [\SL_2/\bG_m]}(R))$,
which gives the splitting of \cite{mhs} Theorem \ref{mhs-mtsmal} on pulling back along $[\rho]$.
\end{proposition}
\begin{proof}
The proof of \cite{mhs} Corollary \ref{mhs-formalitysl} carries over, by using the principle of two types from Lemma \ref{typesuniv}.
\end{proof}

\section{Unitary actions}\label{u1}

We now seek to describe additional structure on $X_{\MTS}^{R, \univ}$ generalising the $U_1$-action of \cite{mhs} Proposition \ref{mhs-mtsmalenrich}. The $\dmd$-action of $\lambda \in S(\R)$ on  $A^n(X)$ gives a discrete action on $\tilde{A}^{\bt}(X,O)$, when combined with the standard action on $C^*$ and the $\clubsuit$-action of $\frac{\bar{\lambda}}{\lambda}$ on $\fM_{\cH,X,R}$ (this is essentially the $\spadesuit$-action of \cite{mhs} Proposition \ref{mhs-mhsmal}). However, we wish to capture the analytic properties of this action. Informally, we would like to work on the site
$$
([\fM_{\cH,X,R}/\langle \tau \rangle] \by C^*)/\iota S,
$$
with $\iota:S \to U_1 \by S$ (given by $\lambda \mapsto (\frac{\bar{\lambda}}{\lambda}, \lambda)$) acting on both terms, as in the description above. The problem with this is that  $[\fM_{\cH,X,R}/\langle \tau \rangle]$ is analytic, with an analytic $U_1$-action, while $C^*$ is algebraic, with an algebraic $S$-action. 

We circumvent this by considering that for an immersion $G \into H$ of group schemes, with $H$ acting on an algebraic stack $\fN$, there are canonical isomorphisms
$$
[\fN/G] \cong [\fN/H]\by_{BH}BG \cong [\fN/H]\by_{BH}[(G/H)/H],
$$
so 
 quasi-coherent sheaves on $[\fN/G]$ correspond to  quasi-coherent  $q^*\O_{[(G/H)/H]}$-modules on $[\fN/H]$, for $q:[\fN/H]\to BH$.  \cite{mhs} Lemma \ref{mhs-tfilen} is a special case of this, with $\fN=C^*$, $G=\bG_m$ and $H=S$.

Now, the quotient of the map $\iota$ given above  is isomorphic to $S/\bG_m \cong U_1$, motivating the following definition.

\begin{definition}
Observe that on the product $B(U_1^{\an})\by BS$ of an analytic stack with an algebraic stack, quasi-coherent sheaves correspond to vector spaces equipped with an analytic $U_1$-action commuting with an algebraic $S$-action.

Define $\sU$ to be the quasi-coherent sheaf on this site corresponding to the real algebra $O(U_1)$ (of polynomial functions on $U_1$), equipped with its usual structure $U_1$-action, together with an $S$-action in which $\lambda$ acts as  $\frac{\lambda}{\bar{\lambda}}$. Note that the image of $\iota$ corresponds to the kernel of these actions combined. 
\end{definition}

\begin{definition}
For the quotient map $q: [\fM_{\cH,X,R}/(\langle \tau \rangle \by U_1^{\an})] \to B(U_1^{\an})\by BS$, define 
$$
DG_{\Z}\Alg_{([\fM_{\cH,X,R}/\langle \tau \rangle] \by C^*)/\iota S}(R):= (q^*\sU) \da  DG_{\Z}\Alg_{[\fM_{\cH,X,R}/(\langle \tau \rangle \by U_1^{\an})] \by [C^*/ S]}(R).
$$
\end{definition}

\begin{lemma}\label{univfchar}
To give an object of $A \in DG_{\Z}\Alg_{([\fM_{\cH,X,R}/\langle \tau \rangle] \by C^*)/\iota S}(R)$ is equivalent to   giving, for all  $\tau, U_1$-equivariant smooth maps $f:V \to \fM_{\cH,X,R}$ a  of analytic stacks, and all  smooth $S$-equivariant maps $g: W \to C^*$  of algebraic spaces, the compatible objects
$$
A|_{V\by W} \in DG_{\Z}\Alg_{U\by V}(R),
$$
 equipped with their  natural $S^{\delta}=\Cx^*$-representations,  subject to the conditions that 
\begin{enumerate}
\item the action of $\bG_m \subset S$ is algebraic, and compatible with the $\bG_m$-action on $\Cx^*$; 
\item the action of $U_1 \subset S$ is analytic, and compatible with the  $U_1^{\delta}$-action on $ \fM_{\cH,X,R} \by \Cx^*$ given by $t\spadesuit (y,c)= (t^{-2}\clubsuit y, t c)$.\end{enumerate}
\end{lemma}
\begin{proof}
We give the construction of $A|_{V\by W}$. The object $A$ corresponds to some 
$$
B \in (q^*\sU) \da  DG_{\Z}\Alg_{[\fM_{\cH,X,R}/(\langle \tau \rangle \by U_1^{\an})] \by [C^*/ S]}(R),
$$ 
by definition. The site $[\fM_{\cH,X,R}/(\langle \tau \rangle \by U_1^{\an})] \by [C^*/ S]$ is generated by spaces of the form $[V/\langle \tau \rangle \by U_1^{\an}]\by [W/S]$, so we may consider $(f,g)^*B \in (q^*\sU) \da DG_{\Z}\Alg_{V\by W}(R)$. This has an analytic $U_1$-action over $W$, and an algebraic $S$-action over $W$.

Now, the unit $1 \in U_1$ gives a map $O(U_1) \to \R$, and hence an  $\iota S$-equivariant map $q^*\sU \to \O_{V \by W}$, so we set
$$
A|_{V\by Q}:= (f,g)^*B\ten_{\sU} \O_{V \by W},
$$
noting that this has an $S$-action given by $\iota$, and satisfying the required properties.
\end{proof}

\begin{proposition}\label{univmtsen}
Given a compact K\"ahler manifold $X$, and a real pro-algebraic groupoid $R$ with a Cartan involution $C$,  
there is a canonical object 
$$
X_{\MHS}^{R, \univ} \in \Ho(dg_{\Z}\Aff_{[\bA^1/\bG_m]\by [([\fM_{\cH,X,R}/\langle \tau \rangle] \by C^*)/\iota S]}(R)).
$$
 Pulling back along 
$$
[\fM_{\cH,X,R}/\langle \tau \rangle] \by [C^*/\bG_m] = [([\fM_{\cH,X,R}/\langle \tau \rangle] \by C^*)/\iota \bG_m ]\xra{\vareps} [([\fM_{\cH,X,R}/\langle \tau \rangle] \by C^*)/\iota S]
$$
gives $\vareps^*X_{\MHS}^{R, \univ}= X_{\MTS}^{R, \univ}$.
\end{proposition}
\begin{proof}
We use the criteria of  Lemma \ref{univfchar}. We will enhance the structure of $\tilde{A}^{\bt}(X,O)$ from Theorem \ref{univmts} to give an object of $DG\Alg_{([\fM_{\cH,X,R}/\langle \tau \rangle] \by C^*)/\iota S}(R)$. For $V,W$ as in Lemma \ref{univfchar}, consider 
$$
\Gamma({V \by W}, \tilde{A}^{n}(X, O)),
$$ 
and observe that the  analytic $U_1$-action on $V$ allows us to extend the algebraic $\bG_m$-action of Theorem \ref{univmts} to an analytic $S$-action, given by the  $\spadesuit$ formula. This satisfies the conditions of  Lemma \ref{univfchar}, allowing us to define $X_{\bF}^{R, \univ} \in \Ho(dg_{\Z}\Aff_{[([\fM_{\cH,X,R}/\langle \tau \rangle] \by C^*)/\iota S]}(R))$. Taking the Rees algebra construction gives $X_{\MHS} ^{R, \univ}$.
\end{proof}

\begin{remarks}\label{cfmhs}
\begin{enumerate}
\item
Given a $U_1^{\an}$-equivariant map $[\rho]:\Spec \R \to [\fM_{\cH,X,R}/\langle \tau \rangle]$, we may pull back $X_{\MHS}^{R, \univ}$ to obtain an object 
$$
 [\rho]^* X_{\MHS}^{R, \univ}\in \Ho(dg_{\Z}\Aff_{BU_1^{\an} \by [\bA^1/\bG_m]\by [C^*/S]}(R) \da (\Spec \sU)\by [C^*/S]  ),
$$ 
which is well-defined since $O(X_{\MHS}^{R, \univ})$ is flat and bounded. If $[\rho]$ comes from a Zariski-dense representation $\rho$, then observe that $U_1$-equivariance amounts to giving a  homomorphism $\alpha:U_1^{\an} \to R^C$ such that $t\clubsuit \rho= \ad_{\alpha(t)}\rho$. This gives a $U_1$-action (and hence an $S$-action) on $R$, as in \cite{mhs} Theorem \ref{mhs-mhsmal}. There is then an isomorphism $  R\rtimes S \cong R\by S$, given by  $(r,s) \mapsto (r\alpha(\frac{\bar{s}}{s}), s)$. We may thus regard $X_{\MHS}^{\rho,\mal}$ as an object of 
\begin{eqnarray*}
\Ho(dg_{\Z}\Aff_{\bA^1 \by C^*}(\bG_m \by R\by  S)) &\simeq& \Ho(dg_{\Z}\Aff_{[\bA^1/\bG_m] \by [C^*/S]}(R))\\ &\simeq& \Ho(dg_{\Z}\Aff_{ BU_1 \by[C^*/S]}(R) \da \Spec \sU)\by [C^*/S]).
\end{eqnarray*} 
Pulling back along $BU_1^{\an} \to BU_1$ (i.e. taking the forgetful functor from algebraic to analytic $U_1$-representations), this recovers $ [\rho]^* X_{\MHS}^{R, \univ}$.

\item
For any  map (not necessarily $U_1$-equivariant)  $[\rho]:\Spec \R \to [\fM_{\cH,X,R}/\langle \tau \rangle]$ coming from a Zariski-dense representation $\rho$,  the pullback of $X_{\MHS}^{R, \univ} $ along $[\rho]$ and $1 \in C^*$ gives $X^{\rho,\mal}$, by Theorem \ref{univmts}.
 The action of $U_1\subset S$  then gives us an analytic map
$$
X^{\rho,\mal} \by U_1^{\an} \to X_{\MHS}^{R, \univ}\by_{[([\fM_{\cH,X,R}/\langle \tau \rangle] \by C^*)/\iota S]}[\fM_{\cH,X,R}/\langle \tau \rangle] \by C^*;
$$
this is essentially a generalisation of the analytic $U_1$-action of \cite{mhs} Proposition \ref{mhs-mtsmalenrich}.

\end{enumerate}
\end{remarks}

We also have the following extension of \cite{mhs} Corollary \ref{mhs-formalitysl}, showing that there is an analytic $S$-action on $\ugr X_{\MTS}^{R, \univ}$ (whose induced $\bG_m$-action is algebraic), and a splitting of $ X_{\MHS}^{R, \univ}$ over $\SL_2$.
\begin{proposition}\label{formalityuniv2}
There is a canonical object $\ugr X_{\MHS}^{R, \univ}$  of   
$$
\Ho(dg_{\Z}\Aff_{[\fM_{\cH,X,R}/\langle \tau \rangle\by\iota S]}(R)):=   \Ho(dg_{\Z}\Aff_{[\fM_{\cH,X,R}/(\langle \tau \rangle \by U_1^{\an})] \by BS}(R)\da\Spec(q^*\sU)),
$$ 
with $\vareps^*\ugr X_{\MHS}^{R, \univ}=\ugr X_{\MTS}^{R, \univ} $, for $\vareps$ as in Proposition \ref{univmtsen}.

There is also an isomorphism 
$$
 \row_1^*X_{\MHS}^{R, \univ} \cong \ugr X^{R, \univ}_{\MHS}\by_{[\fM_{\cH,X,R}/\langle \tau \rangle\by \iota S]}  [([\fM_{\cH,X,R}/\langle \tau \rangle] \by \SL_2)/\iota S].
$$ 
extending the isomorphism of Proposition \ref{formalityuniv}.
\end{proposition}
\begin{proof}To define $\ugr X_{\MTS}^{R, \univ}$, adapt the proof of Proposition \ref{univmtsen}, replacing $\tilde{A}^{\bt}(X,O)$ with $A^{\bt}(X,O)$. The proof of Proposition \ref{formalityuniv} then adapts to give the splitting isomorphism.
\end{proof}

\bibliographystyle{alphanum}
\addcontentsline{toc}{section}{Bibliography}
\bibliography{references}

\end{document}